\documentclass[11pt]{article}

\usepackage{latexsym}
\usepackage{amsmath}
\usepackage{amsfonts}
\usepackage{amssymb}
\usepackage{amsthm}
\usepackage{psfrag}
\usepackage{epsfig}
\usepackage{color}

\usepackage[tight]{subfigure}
\newcommand{\R}{\mathbb{R}}
\newcommand{\T}{\Theta}
\newcommand{\one}{\mathbf{1}}
\newcommand{\Bu}{\mathbf{u}}
\newcommand{\Bv}{\mathbf{v}}
\newcommand{\<}{\langle}
\renewcommand{\>}{\rangle}
\newcommand{\jump}[1]{{[{#1}]}}

\newcommand{\rref}[1]{{\rm\ref{#1}}}

\newcommand{\G}{\Gamma}

\newcommand{\CT}{{\cal T}}
\newcommand{\CV}{{\cal V}}
\newcommand{\CW}{{\cal W}}

\newcommand{\be}{\begin{equation}}
\newcommand{\ee}{\end{equation}}

\newtheorem{theorem}{Theorem}[section]

\newtheorem{remark}[theorem]{Remark}
\newtheorem{lemma}[theorem]{Lemma}
\newtheorem{corollary}[theorem]{Corollary}

\title{
Ultra-weak formulation of a hypersingular integral equation on polygons
and DPG method with optimal test functions
\thanks{Supported by FONDECYT project 1110324 and
        CONICYT project Anillo ACT1118 (ANANUM).}
}

\author{
Norbert Heuer
\thanks{
Facultad de Matem\'aticas, Pontificia Universidad Cat\'olica de Chile,
Avenida Vicu\~na Mackenna 4860, Macul, Santiago, Chile,
email: {\tt nheuer@mat.puc.cl}}
\and
Felipe Pinochet
\thanks{
Facultad de Ingenier\'\i a, Pontificia Universidad Cat\'olica de Chile,
Avenida Vicu\~na Mackenna 4860, Macul, Santiago, Chile,
email: {\tt fipinoch@gmail.com}}
}

\begin{document}
\date{}
\maketitle

\bigskip
\begin{abstract}
We present an ultra-weak formulation of a hypersingular integral equation
on closed polygons and prove its well-posedness and equivalence with the
standard variational formulation. Based on this ultra-weak formulation we present
a discontinuous Petrov-Galerkin method with optimal test functions and
prove its quasi-optimal convergence in $L^2$. Theoretical results are confirmed
by numerical experiments on an open curve with uniform and adaptively refined meshes.


\end{abstract}

\section{Introduction}

The design and analysis of numerical methods for the solution of hypersingular
integral equations is inherently difficult due to the nature of their underlying
solution spaces. In case of the operator related to the Laplacian, this space
is the trace of $H^1$ functions onto the boundary of the problem. It is the
Sobolev space $H^{1/2}$ of order $1/2$ whose norm is non-local by nature.
In this paper we present an ultra-weak formulation that allows to avoid using
fractional-order Sobolev spaces for its setting and for that of subsequent
Petrov-Galerkin approximations. While this change of Sobolev spaces is
mathematically useful there are also several practical implications as is known
from corresponding setups for partial differential equations.

To our knowledge, the use of ultra-weak formulations in the numerical approximation
of partial differential equations has started with the works by Despr\'es
and Cessenat \cite{Despres_94_SUF,CessenatD_98_AUW}. Recently, Demkowicz and
Gopalakrishnan have proposed a discontinuous Petrov-Galerkin (DPG)
method based on ultra-weak formulations where norms and test functions are tailored
towards stability \cite{DemkowiczG_11_CDP,DemkowiczG_11_ADM}. This is particularly
useful for singularly perturbed problems like convection-dominated diffusion
and wave problems \cite{ZitelliMDGPC_11_CDP,DemkowiczH_RDM}.

In this paper, we consider the model problem of the hypersingular integral equation
on polygons which appears when dealing with the Laplacian on a polygonal domain
and Neumann boundary condition. This integral equation gives rise to a well-posed
variational formulation in the trace space of $H^1$ functions which can be solved
by boundary elements, cf.~\cite{NedelecP_73_MVE,HsiaoW_77_FEM} and see
\cite{Stephan_86_BIE} for open surfaces. In recent years there has been some progress
in the use of discontinuous approximations for hypersingular operators,
in particular Crouzeix-Raviart elements \cite{HeuerS_09_CRB},
mortar and Nitsche domain decompositions \cite{HealeyH_10_MBE,ChoulyH_12_NDD},
and discontinuous $hp$ approximation \cite{HeuerM_DGB}.
These discontinuous approximations constitute a variational crime since an analysis
of their jumps can not be based on (well-defined) trace operators in the Sobolev space
of order $1/2$. As a consequence some of the results of these papers suffer from
logarithmic perturbations which lead to sub-optimal error estimates, and others are
based on mesh and penalty parameter restrictions.

The method we propose in this paper does not suffer from such variational crimes.
Instead of analyzing a weak form of the hypersingular operator in fractional-order
Sobolev spaces (based on $H^{1/2}$) we propose an ultra-weak formulation which
shifts trial and test functions away from $H^{1/2}$ to $L^2$ and piecewise $H^1$
spaces. We then study the DPG method with optimal test functions based on the
ultra-weak formulation and corresponding trial and test spaces with appropriate
norms. The presentation follows \cite{DemkowiczH_RDM} and puts special emphasis
on the analysis of norms in the test space. In particular, practicality of the
method requires that test norms are localizable. Having well-posedness of the
ultra-weak formulation and appropriate norm equivalences in the test space,
the (quasi-) optimality of the DPG approximation with optimal test functions
follows by standard arguments. We do not repeat these ideas and arguments here
but refer to \cite{DemkowiczG_11_CDP,ZitelliMDGPC_11_CDP} for details.
The setting of the DPG method with optimal test functions will be given in
Section~\ref{sec_DPG} below.

Whereas our ultra-weak formulation avoids the setting of fractional-order Sobolev
spaces it obviously involves a boundary integral operator. Therefore, norms
are localizable but the problem under consideration is still global. This
global effect enters through the calculation of optimal test functions. In contrast
to DPG methods with optimal test functions for partial differential equations,
we cannot calculate test functions on the fly. Nevertheless,
\begin{itemize}
\item the linear systems of the DPG method, for approximating optimal test functions
and for error calculation have sparse matrices.
\end{itemize}
This is due to the use of localizable norms in the test space.
In the case of our model problem, some parts of the test functions can be given
analytically and one component has to be approximated
(using a weakly singular rather than hypersingular operator).
Other standard advantages of the method are maintained:
\begin{itemize}
\item Error control is inherent since errors in the energy norm (which is now
      $L^2$ rather than $H^{1/2}$) can be calculated through the implementation
      of the trial-to-test operator ($\Theta$ in \eqref{Theta}).
\item Since norms are localizable the energy norm of the error gives local
      information which can be used to steer adaptive refinements.
\item Error estimates and stability hold for any combination of meshes and
      polynomial degrees so that $hp$ methods do not require a new analysis.
\item Since approximation spaces are composed of $L^2$ parts and
      trial functions can be discontinuous, one has full flexibility
      for $h$ and $p$ adaptivity.
\end{itemize}

Our analysis makes use of the $H^1$-regularity of the solution
to the model problem. On open curves, however, the solution is not in $H^1$ but
rather an element of Sobolev spaces of any order smaller than one. We suppose that
our method and techniques can be extended to this case by switching back
from $L^2$ bilinear forms to dualities between fractional-order Sobolev spaces
(though of orders close to zero).
This is left to future research. Nevertheless, numerical experiments are
performed for this limit case (on an interval as curve in $\R^2$)
and produce convincing results.

Of course, the advantages of the DPG method with optimal test functions
are more relevant for problems in three space dimensions.
Our setup of an ultra-weak formulation and analytical techniques should
in principle extend to higher dimensions, which is also left to future
research.

The remainder of the paper is as follows. In the next section we define the model
problem and derive an ultra-weak formulation. Theorem~\ref{thm_ultra} proves
its well-posedness and equivalence with the standard variational form.
The proof of stability makes use of the equivalence of different norms in the
test space (Theorems~\ref{thm_equiv1},~\ref{thm_equiv2}).
These equivalences are shown in Section~\ref{sec_norms}, and are
based on a stability analysis in Section~\ref{sec_tech} of the adjoint problem
with respect to the bilinear form of the ultra-weak formulation.
As consequence of equivalences of different test norms we also have equivalences
of corresponding norms in the trial space (Corollary~\ref{cor_equiv}).
This is essential to prove quasi-optimal convergence of the DPG method with optimal
test functions. The method and this result (Theorem~\ref{thm_DPG}) is presented
in Section~\ref{sec_DPG}. The calculation of test functions for our model problem
is done in Section~\ref{sec_test}. In Section~\ref{sec_int} we elaborate on
the differences in the formulation for an open curve, and optimal test functions
are discussed as well. No analysis is given for this case. Numerical results
on an interval are reported in Section~\ref{sec_num}. We end with some conclusions.

Throughout the paper, $a\lesssim b$ means that $a\le cb$ with a generic constant
$c>0$ that is independent of involved parameters like $h$ or $p$.
Similarly, the notation $a\gtrsim b$ and $a\simeq b$ is used.

\section{Model problem and ultra-weak formulation} \label{sec_model}
\setcounter{equation}{0} \setcounter{figure}{0} \setcounter{table}{0}

We present our model problem, an ultra-weak formulation thereof and prove its
well-posedness.

The model problem is as follows.
We consider a polygonal domain (simply connected, Lipschitz) $\Omega\subset\R^2$
with boundary curve $\G$ and assume that $\G$ has logarithmic capacity
cap$(\G)<1$ (this will be needed for coercivity of the weakly singular operator
introduced below, cf.~\cite[page 264]{McLean_00_SES}).
Then, for given $f\in L^2_0(\G)$, we look for $\phi\in H^{1/2}(\G)/\R$ such that
\be \label{model}
   \CW\phi=f\qquad \text{on}\quad \G.
\ee
Here, $\CW$ is the hypersingular operator defined by
\[
   \CW v =
   \frac{1}{2\pi} \frac{\partial}{\partial n}\int_\G v(y) \frac{\partial}{\partial n_y}
                                       \log|\cdot-y|\,ds_y,
   \qquad v\in H^{1/2}(\G)
\]
with unit normal vector $n$ exterior to $\Omega$,
\[
   L^2_0(\G):=\{v\in L^2(\G);\; \<v,1\>_\G=0\},
\]
and $\<\cdot,\cdot\>_\G$ refers to the $L^2(\G)$-inner product and its extension
to duality between $H^{-1/2}(\G)$ and $H^{1/2}(\G)$. The space $H^{1/2}(\G)$ is
the trace of $H^1(\Omega)$ and $H^{-1/2}(\G)$ its dual.

Problem \eqref{model} models, for appropriate right-hand side $f$,
the Laplace equation with Neumann boundary condition
within or exterior to $\Omega$. A standard variational formulation of \eqref{model} is
\be \label{weak}
   \phi\in H^{1/2}(\G)/\R:\qquad
   \<\CW\phi,\psi\>_\G = \<f, \psi\>_\G\quad\forall \psi\in H^{1/2}(\G).
\ee
The bilinear form of the hypersingular operator is usually implemented
by making use of the relation
\[
   \<\CW v,\psi\>_\G = \<\psi',\CV v'\>_\G\qquad (v,\psi\in H^{1/2}(\G))
\]
which can be written as a relation between linear functionals in $H^{-1/2}(\G)$ like
\be \label{WV}
   \CW v = -\bigl(\CV v'\bigr)'\qquad (v\in H^{1/2}(\G)),
\ee
cf.~\cite{Nedelec_82_IEN}, see also \cite{CostabelES_91_ECR} for details.
Here, $(\cdot)'$ denotes differentiation with respect to the arc length and
$\CV$ is the weakly-singular operator defined by
\[
   \CV v =
   -\frac{1}{2\pi} \int_\G v(y) \log|\cdot-y|\,ds_y,
   \qquad v\in H^{-1/2}(\G).
\]
There hold the mapping properties \cite{Costabel_88_BIO}
\be \label{V_cont}
   \CV:\; H^{s-1/2}(\G)\to H^{s+1/2}(\G),\qquad s\in [-1/2,1/2].
\ee
Here, the fractional order Sobolev spaces can be defined, e.g., by interpolation,
cf.~\cite{McLean_00_SES} for details.
Our DPG method will be based on an ultra-weak formulation of \eqref{model},
which we derive next.

Using \eqref{WV} we first introduce another unknown $\sigma$ to rewrite
\eqref{model} as the system
\be \label{system}
   \sigma=\CV\phi',\qquad -\sigma'=f.
\ee
This system is considered in a weak form. For the first equation, we test with
$\tau\in L^2(\G)$, use the symmetry of $\CV$ and integrate by parts to obtain
\be \label{sigma_weak}
   \<\sigma,\tau\>_\G = \<\phi',\CV\tau\>_\G = -\<\phi,(\CV\tau)'\>_\G.
\ee
Here we used that $\phi\in H^s(\G)$ for $s\in (1/2,1)$ \cite{StephanW_84_AGP} and
$\CV\tau\in H^1(\G)$ by \eqref{V_cont} so that both functions are continuous by
the Sobolev embedding theorem.

Now, for a weak form of the second identity in \eqref{system}, we test
with piecewise $H^1$-functions. To this end let $\CT$ be a mesh of elements $T$
on $\G$ (and with nodes $x_j$, $j=1,\ldots,N$, and the convention that $x_0:=x_N$)
which is compatible with the geometry (vertices of $\G$ are nodes), and define
\[
   H^1(\CT) = \{v\in L^2(\G);\; v|_T\in H^1(T)\ \forall T\in\CT\}
\]
with norm
\[
   \|v\|_{H^1(\CT)}
   :=
   \Bigl(\|v\|_{L^2(\G)}^2 + \|v'\|_{L^2(\CT)}^2\Bigr)^{1/2}
   :=
   \Bigl(\|v\|_{L^2(\G)}^2 + \sum_{T\in\CT} \|v'\|_{L^2(T)}^2\Bigr)^{1/2}.
\]
Below, we will also need the lengths of the shortest element, $h_{\min}$,
and of the longest element, $h_{\max}$.

For any $v\in H^1(\CT)$, the second identity in \eqref{system} and integration
by parts imply that
\be \label{f_weak}
   \<f,v\>_\G = -\<\sigma',v\>_\G
   = \sum_{T\in\CT} \<\sigma,v'\>_T + \sum_{j=1}^N \sigma(x_j)\jump{v}_j.
\ee
Here, $\jump{v}_j$ denotes the jump of $v$ at $x_j$. More precisely,
let $T_{j-1}$ and $T_j$ be the elements of $\CT$ that are before and after $x_j$,
respectively, in mathematically positive orientation of $\G$. Then
\[
   \jump{v}_j := v|_{T_j}(x_j) - v|_{T_{j-1}}(x_j).
\]
We also use the notation
\[
   \jump{v} := (\jump{v}_j)_{j=1}^N\in\R^N \qquad\text{for}\quad v\in H^1(\CT)
\]
and
\[
   \<\psi,v'\>_\CT := \sum_{T\in\CT} \<\psi,v'\>_T
   \qquad\text{for}\quad \psi\in L^2(\G),\; v\in H^1(\CT).
\]
Below, $\sigma$ in \eqref{f_weak} will be considered as an element of
$L^2(\G)$. Of course, nodal values of $\sigma$ are then not well defined.
New unknowns $\hat\sigma\in\R^N$ are introduced to replace them.

We now combine \eqref{sigma_weak} and \eqref{f_weak} to define our ultra-weak
formulation of \eqref{model}:
{\em Find $\sigma\in L^2(\G)$, $\phi\in L^2(\G)$ and $\hat\sigma\in\R^N$
such that}
\begin{alignat}{5}
   \<\sigma, \tau\>_\G &+ \<\phi, (\CV\tau)'\>_\G && &&= 0
      &\qquad&\forall\tau\in L^2(\G),
   \label{ultra1}\\
   \<\sigma, v'\>_\CT &+ \quad\hat\sigma\cdot\jump{v}
      &&+\<\phi,1\>_\G \<v,1\>_\G &&= \<f, v\>_\G
      &\qquad&\forall v\in H^1(\CT).
   \label{ultra2}
\end{alignat}
Here, we have added the rank-one term $\<\phi,1\>_\G \<v,1\>_\G$ to make
$\phi$ unique. Indeed, \eqref{sigma_weak} and \eqref{f_weak} have a kernel
consisting of constants on $\G$ with respect to $\phi$ and $v$, respectively.
Recall that $f\in L^2_0(\G)$.
The additional term selects $\phi$ with integral-mean zero.

Let us formulate the ultra-weak formulation \eqref{ultra1}, \eqref{ultra2} as:
{\em Find $(\sigma, \phi, \hat\sigma) \in L^2(\G)\times L^2(\G)\times\R^N$
such that}
\be \label{ultra}
   b(\phi,\sigma,\hat\sigma;\tau,v) = \<f,v\>_\G
   \qquad\forall (\tau,v)\in L^2(\G)\times H^1(\CT).
\ee
Here,
\[
   b(\phi,\sigma,\hat\sigma;\tau,v)
   := \<\phi, (\CV\tau)'\>_\G + \<\phi,1\>_\G \<v,1\>_\G
   + \<\sigma, \tau+v'\>_\CT + \hat\sigma\cdot\jump{v}.
\]
This formulation is designed so that subsequent Petrov-Galerkin methods provide
best approximations of $\sigma$ and $\phi$ in $L^2(\G)$. In particular, the solution
$\phi$ of \eqref{model} can be approximated by discontinuous functions whereas
conforming approximations (piecewise polynomials) based on \eqref{weak} must be continuous.
The Petrov-Galerkin bilinear form $b(\cdot,\cdot)$ and our interest to control
$\sigma$ and $\phi$ in $L^2(\G)$ suggest to consider the norm
\be \label{U}
   \|(\phi,\sigma,\hat\sigma)\|_{U,\alpha} :=
   \|\phi\|_{L^2(\G)} + \|\sigma\|_{L^2(\G)} + \alpha |\hat\sigma|
\ee
($|\cdot|$ is the Euclidean norm in $\R^N$ and $\alpha>0$ has to be selected)
in the solution space
\[
   U:=L^2(\G)\times L^2(\G)\times \R^N.
\]
The corresponding so-called optimal test norm in the test space
\[
   V:=L^2(\G)\times H^1(\CT)
\]
is
\begin{align} \label{Vopt}
   \|(\tau,v)\|_{V,{\rm opt},\alpha}
   &:=
   \sup_{(\phi,\sigma,\hat\sigma)\in U\setminus\{0\}}
   \frac {b(\phi,\sigma,\hat\sigma;\tau,v)}{\|(\phi,\sigma,\hat\sigma)\|_{U,\alpha}}
   \nonumber
   \\
   &\simeq
     \|(\CV\tau)'\|_{L^2(\G)}
   + \|\tau+v'\|_{L^2(\CT)}
   + \alpha^{-1} |\jump{v}| + |\<v,1\>_\G|.
\end{align}
The equivalence above is immediate from the definition of the bilinear form
$b(\cdot,\cdot)$, and we recall that $\|\tau+v'\|_{L^2(\CT)}$ indicates
that $v'$ is meant in a piecewise sense with respect to the mesh $\CT$.

\begin{remark}
Theorems~\rref{thm_equiv1} as well as \rref{thm_equiv2} (shown below) confirm that
\eqref{Vopt} is indeed a norm. This is essential for proving stability
of the ultra-weak formulation and for the analysis of the DPG method.
\end{remark}

The following theorem is one of our main results and forms the basis for
our DPG method with optimal test functions to solve the hypersingular integral
equation \eqref{model}.

\begin{theorem} \label{thm_ultra}
There exists a unique solution to the ultra-weak formulation
\eqref{ultra} which is stable in the sense that
\be \label{stab}
   \|(\phi,\sigma,\hat\sigma)\|_{U,\alpha}
   \lesssim
   \|f\|_{L^2(\G)}
   \qquad\text{with } \alpha=N^{-1/2}.
\ee
Furthermore, the ultra-weak formulation and
the standard weak formulation \eqref{weak} are equivalent. More precisely,
if $\phi_\R$ solves \eqref{weak} then, with $|\G|$ being the length of $\G$,
$\sigma:=\CV\phi_\R'$, $\phi:=\phi_\R-|\G|^{-1}\<\phi_\R,1\>_\G$ and
$\hat\sigma:=(\CV\phi_\R'(x_j))_{j=1}^N$ solve \eqref{ultra}.
If $(\phi,\sigma,\hat\sigma)$ solves the ultra-weak formulation then $\phi$
solves \eqref{weak}.
\end{theorem}

\begin{proof}
For the time being, let us define the so-called energy norm in $U$ by
\be \label{E}
   \|(\phi,\sigma,\hat\sigma)\|_{E,\alpha}
   :=
   \sup_{(\tau,v)\in V\setminus\{0\}}
   \frac {b(\phi,\sigma,\hat\sigma;\tau,v)}{\|(\tau,v)\|_{V,\mathrm{opt},\alpha}}.
\ee
First we show that this is indeed a norm, i.e.,
that $\|(\phi,\sigma,\hat\sigma)\|_{E,\alpha}=0$ implies $(\phi,\sigma,\hat\sigma)=0$.
Then existence, uniqueness and stability of the ultra-weak solution
in this norm follow by standard Babu\v{s}ka-Brezzi theory, by showing that
\begin{align}
   \label{bounded}
   b(\cdot,\cdot):\; &
   (U, \|\cdot\|_{E,\alpha})\times (V, \|\cdot\|_{V,\mathrm{opt},\alpha})
   \to \R \quad\text{is bounded},\\
   \label{infsup}
   \sup_{(\tau,v)\in V\setminus\{0\}}
   & \frac {b(\phi,\sigma,\hat\sigma;\tau,v)}{\|(\tau,v)\|_{V,\mathrm{opt},\alpha}}
   \gtrsim \|(\phi,\sigma,\hat\sigma)\|_{E,\alpha} \qquad\forall (\phi,\sigma,\hat\sigma)\in U,\\
   \label{pos}
   \sup_{(\phi,\sigma,\hat\sigma)\in U\setminus\{0\}}
   & \frac {b(\phi,\sigma,\hat\sigma;\tau,v)}{\|(\phi,\sigma,\hat\sigma)\|_{E,\alpha}}
   > 0 \qquad\forall (\tau,v)\in V\setminus\{0\},
\end{align}
and verifying boundedness of the right-hand side functional.
\begin{enumerate}
\item We show that $\|\cdot\|_{E,\alpha}$ is a norm.
Let $(\phi,\sigma,\hat\sigma)\in U$ be such that
$b(\phi,\sigma,\hat\sigma;\tau,v)=0$ for any $(\tau,v)\in V$.
Selecting $v=1$ it is clear that $\<\phi,1\>_\G=0$.
From \eqref{ultra2} we conclude that, for any $T\in\CT$, there holds
\[
   \<\sigma, v'\>_T = 0\qquad\forall v\in H^1_0(T).
\]
Therefore, $\sigma'=0$ in distributional sense on every element, i.e., $\sigma\in H^1(\CT)$.
Now let $T\in\CT$ be given with endpoints $x_j$ and $x_{j+1}$. Considering $v\in H^1(\CT)$
with support in $\bar T$, using \eqref{ultra2}, and integrating by parts, we obtain
\[
   \bigl(\sigma(x_{j+1})-\hat\sigma_{j+1}\bigr)v(x_{j+1}) -
   \bigl(\sigma(x_j)-\hat\sigma_j\bigr)v(x_j) = 0
   \qquad\forall v\in H^1(T).
\]
This implies $\sigma(x_j)=\hat\sigma_j$, $j=1,\ldots,N$, i.e., $\sigma\in H^1(\G)$ and
$\sigma$ is a constant.

Now, $\sigma$ being a constant, \eqref{ultra1} implies that
\[
   \<\phi,(\CV\tau)'\>_\G = 0 \qquad\forall\tau\in L^2_0(\G).
\]
By \cite[Lemma 8.14]{McLean_00_SES} there exists for any $\Psi\in H^1(\G)$
a unique $\tau\in H^{-1/2}(\G)$ and $a\in\R$
such that $\<\tau,1\>_\G=0$ and $\CV\tau=\Psi+a$.
Moreover, by \cite[Theorem 3]{Costabel_88_BIO}, $\tau\in L^2(\G)$.
Therefore, since the derivative operator $(\cdot)'$ maps $H^1(\G)$ onto
$L^2_0(\G)$, the mapping
\[
   (\CV\cdot)':\; L^2_0(\G)\to L^2_0(\G)
\]
is onto. We conclude that
\[
   \<\phi,\psi\>_\G = 0
   \qquad\forall\psi\in L^2_0(\G)
\]
so that, since $\<\phi,1\>_\G=0$, $\phi=0$.
Then by \eqref{ultra1}, $\<\sigma,\tau\>_\G=0$ for any $\tau\in L^2(\G)$ so that $\sigma=0$.
Since $\sigma=\hat\sigma_j$, $j=1,\ldots,N$,
this proves that $(\sigma,\phi,\hat\sigma)=0$, i.e. uniqueness of a
solution to the ultra-weak formulation and definiteness of the energy norm.
\item 
The properties \eqref{bounded} and \eqref{infsup} (with constant $1$) are immediate by
the definition of the energy norm \eqref{E}.
\item We show \eqref{pos} by proving that
\be \label{pf_pos}
   (\tau,v)\in V:\quad b(\phi,\sigma,\hat\sigma; \tau, v)=0
   \quad\forall (\phi,\sigma,\hat\sigma)\in U
\ee
implies that $(\tau,v)=0$.

Indeed, testing in \eqref{pf_pos} separately with $\phi=1$ (and $\sigma=0$, $\hat\sigma=0$),
then with $\sigma\in L^2(\G)$, and $\hat\sigma\in\R^N$ (other functions zero) we find that
\[
   \<v,1\>_\G=0,\quad \tau+v'=0\ \text{in}\ L^2(\CT),
                \quad \jump{v}=0
\]
so that, in particular, $v\in H^1(\G)$ and $v'=-\tau$. Then, testing in \eqref{pf_pos} with
$\phi=v$ and integrating by parts, we obtain
\[
   0=\<v,(\CV\tau)'\>_\G = -\<v',\CV\tau\>_\G = \<\tau,\CV\tau\>_\G.
\]
Since $\<\tau,1\>_\G=-\<v',1\>_\G=0$ and since $\CV$ is elliptic on the subspace
of $H^{-1/2}(\G)$ functionals with integral-mean zero \cite[Theorem 8.12]{McLean_00_SES},
we conclude that $\tau=0$. By the previous relations for $v$ we also obtain $v=0$.
This proves \eqref{pos}.
\item
Now we verify boundedness of the linear functional
$V\ni(\tau,v)\mapsto \<f,v\>_\G$ with respect
to the optimal test norm when $\alpha=N^{-1/2}$. 
To this end we make use of the bounds \eqref{PF2} and \eqref{equiv2_2} below
(the norm $\|\cdot\|_{V,2}$ is defined in \eqref{V2}) to deduce that
\[
   \|v\|_{L^2(\G)} \lesssim \|(0,v)\|_{V,\mathrm{opt},\alpha} \quad\forall v\in H^1(\CT)
\]
holds for $\alpha=N^{-1/2}$. Therefore, since $f\in L^2(\G)$ by assumption,
the linear functional is bounded as wanted.
\end{enumerate}

In conclusion we obtain stability in the form
\be \label{pf_stab}
   \|(\phi,\sigma,\hat\sigma)\|_{E,\alpha} \lesssim \|f\|_{L^2(\G)}
   \qquad (\alpha=N^{-1/2}).
\ee
Let us also note that, as consequence of properties \eqref{bounded}--\eqref{pos}
(with boundedness and inf-sup constants $1$), the operator
\[
   B:\; \left\{\begin{array}{clc}
      (U, \|\cdot\|_{E,\alpha}) & \to &(V, \|\cdot\|_{V,\mathrm{opt},\alpha})'\\
      (\phi,\sigma,\hat\sigma)  & \mapsto & b(\phi,\sigma,\hat\sigma; \cdot)
   \end{array}\right.
\]
is an isometric isomorphism, and so is
\[
  B':\; (V, \|\cdot\|_{V,\mathrm{opt},\alpha}) \to (U, \|\cdot\|_{E,\alpha})'
\]
when identifying $V''$ and $V$, cf.~\cite{Demkowicz_06_BB}.
This in turn implies (see~\cite{ZitelliMDGPC_11_CDP} for details,
in particular Proposition~2.1) that the norms $\|\cdot\|_{U,\alpha}$ and $\|\cdot\|_{E,\alpha}$
in $U$ are identical, that is
\be \label{EU}
   \|(\phi,\sigma,\hat\sigma)\|_{U,\alpha}
   =
   \sup_{(\tau,v)\in V\setminus\{0\}}
   \frac {b(\phi,\sigma,\hat\sigma;\tau,v)}{\|(\tau,v)\|_{V,\mathrm{opt},\alpha}}
   \qquad\forall (\phi,\sigma,\hat\sigma)\in U.
\ee
Together with \eqref{pf_stab} this finishes the proof of stability \eqref{stab}.

It is left to show the equivalence of the ultra-weak formulation \eqref{ultra}
and the standard weak form \eqref{weak}.
It is well known that there exists a solution $\phi$ of \eqref{weak} which is
unique in $H^{1/2}(\G)/\R$. Moreover, since $f\in L^2(\G)$, it has regularity
$\phi\in H^1(\G)$, cf.~\cite[Theorem 3]{Costabel_88_BIO}.
Therefore, by the continuity \eqref{V_cont} of $\CV$ there holds
$\sigma:=\CV\phi'\in H^1(\G)$, and $(\sigma,\phi)$ solves \eqref{system}.
Then, following the derivation of \eqref{ultra1}, \eqref{ultra2}, and
defining $\hat\sigma:=(\sigma(x_j))_{j=1}^N$, the triple
$(\sigma,\phi-|\G|^{-1}\<\phi,1\>_\G,\hat\sigma)$ solves the ultra-weak formulation.

Since the standard weak and ultra-weak formulations are both uniquely solvable
their equivalence follows.
\end{proof}

\section{Test norms} \label{sec_norms}
\setcounter{equation}{0} \setcounter{figure}{0} \setcounter{table}{0}

For a practical implementation of the DPG method
the norm $\|\cdot\|_{V,{\rm opt},\alpha}$ is difficult to handle
since the problems for $\tau$ and $v$ do not decouple.
Instead we consider norms which are easier to implement since they decouple and are local:
\begin{align}
  \label{V1}
  \|(\tau,v)\|_{V,1} &:= \|\tau\|_{L^2(\G)} + \|v\|_{H^1(\CT)},
  \\
  \label{V2}
  \|(\tau,v)\|_{V,2}
  &:=
  \|\tau\|_{L^2(\G)} + \|v'\|_{L^2(\CT)} + |\hat v|_h.
\end{align}
Here, $|\hat v|_h$ is the weighted Euclidean norm of point values of $v$ defined by
\be \label{Euclid}
   |\hat v|_h := \Bigl(\sum_{T\in\CT} |T| (v|_T(x_T))^2\Bigr)^{1/2}
\ee
and $x_T$ denotes the endpoint of $T$ which comes first in
mathematically positive orientation of $\G$ (the other endpoint would be fine as well).

The norm $\|\cdot\|_{V,1}$ is the natural choice in $V$ and allows for efficient
numerical approximation of the optimal test functions.
In the current one-dimensional setting, the norm $\|\cdot\|_{V,2}$ even allows
for partially analytical computation of optimal test functions,
see Section~\ref{sec_test} below for details.

As before, the norms in $V$ define corresponding energy norms
in the solution space $U$:
\be \label{duality}
   \|(\phi,\sigma,\hat\sigma)\|_{E,i} :=
   \sup_{(\tau,v)\in V\setminus\{0\}}
   \frac {b(\phi,\sigma,\hat\sigma;\tau,v)}{\|(\tau,v)\|_{V,i}},
   \quad i=1,2.
\ee
In this section we analyze equivalence of the different norms in $V$, then
giving equivalences of different norms in $U$ which are dual to the ones in $V$.

\subsection{Technical results} \label{sec_tech}

As is known from previous papers on the DPG method with optimal test functions,
norm equivalences for trial spaces correspond to norm equivalences in test spaces.
One side of these equivalences reduce to stability analyses of dual problems.
Such stability results are provided in this subsection
(Lemmas~\ref{la_c} and \ref{la_0}). The final Lemma~\ref{la_PF} is a
Poincar\'e-Friedrichs estimate for piecewise $H^1$-functions, and is also
needed for the proof of norm equivalences.

\begin{lemma} \label{la_c}
For given $g_1, g_2\in L^2(\G)$ the system
\be \label{vtau_c}
   \tau+v'=g_1\quad\text{on}\quad\G,\qquad
   (\CV\tau)'=g_2\quad\text{on}\quad\G
\ee
has a unique solution $(\tau,v)\in L^2(\G)\times H^1(\G)/\R$
and there holds
\[
   \|\tau\|_{L^2(\G)} + \|v'\|_{L^2(\G)}
   \lesssim
   \|g_1\|_{L^2(\G)} + \|g_2\|_{L^2(\G)}.
\]
\end{lemma}

\begin{proof}
Using \eqref{WV}, combining both relations in \eqref{vtau_c} and
taking into account the continuity \eqref{V_cont} of $\CV$
for $s=1/2$, we find that
\[
   \CW v = -(\CV v')' = g_2 - (\CV g_1)'\in L^2(\G).
\]
This equation has a unique solution $v\in H^{1/2}(\G)/\R$ with regularity
\[
   |v|_{H^1(\G)} \lesssim \|g_2\|_{L^2(\G)} + \|(\CV g_1)'\|_{L^2(\G)}
                 \lesssim \|g_1\|_{L^2(\G)} + \|g_2\|_{L^2(\G)},
\]
cf. \cite{Costabel_88_BIO}. Here we used again the continuity of $\CV$.
The function $\tau$ is uniquely defined by \eqref{vtau_c} and
by the triangle inequality we can bound
\[
   \|\tau\|_{L^2(\G)} = \|g_1-v'\|_{L^2(\G)}
   \lesssim \|g_1\|_{L^2(\G)} + \|g_2\|_{L^2(\G)}.
\]
\end{proof}

\begin{lemma} \label{la_0}
Any solution $(\tau,v)\in L^2(\G)\times H^1(\CT)$ of
\be \label{vtau_0}
   \tau+v' = 0    \quad\text{in}\quad L^2(\CT),\qquad
   (\CV\tau)' = 0   \quad\text{in}\quad L^2(\G)
\ee
satisfies
\[
   \|\tau\|_{L^2(\G)} + \|v'\|_{L^2(\CT)} \lesssim \sqrt{N} |\jump{v}|.
\]
\end{lemma}

\begin{proof}
The relations \eqref{vtau_0} mean that $a:=\CV\tau$ is a constant and that
\[
   \<\tau,1\>_\G = -\<v',1\>_\CT = \jump{v}\cdot\one
   \qquad\text{with}\quad \one:=(1,\ldots,1)\in\R^N.
\]
Therefore, $\tau$ and $a$ solve
\[
   \tau\in H^{-1/2}(\G),\; a\in\R:\quad
   \CV\tau-a = 0,\quad \<\tau,1\>_\G = \jump{v}\cdot\one.
\]
According to \cite[Lemma 8.14]{McLean_00_SES} its solution is unique.
Furthermore, ellipticity regularity shows that $\tau\in L^2(\G)$. We now
bound $\|\tau\|_{L^2(\G)}$.

The function $\tau$ solves $\CV\tau=a$. Regularity for $\CV$ shows that
\be \label{pf_reg_1}
   \|\tau\|_{L^2(\G)} \lesssim \|\tau\|_{H^{-1/2}(\G)} + \|a\|_{H^1(\G)}
                      = \|\tau\|_{H^{-1/2}(\G)} + \|a\|_{L^2(\G)},
\ee
see \cite{Costabel_88_BIO}. We bound the last two terms. By continuity of $\CV$
(see \eqref{V_cont} with $s=-1/2$)
and the continuous injection $H^{-1/2}(\G)\subset H^{-1}(\G)$ there holds
\be \label{pf_reg_2}
   \|a\|_{L^2(\G)} = \|\CV\tau\|_{L^2(\G)} \lesssim \|\tau\|_{H^{-1}(\G)}
   \lesssim \|\tau\|_{H^{-1/2}(\G)}.
\ee
We are left to bound $\|\tau\|_{H^{-1/2}(\G)}$.
Since cap$(\G)<1$ by assumption, $\CV$ is $H^{-1/2}(\G)$-elliptic,
see \cite[Theorem 8.16]{McLean_00_SES}. Therefore,
by relations \eqref{vtau_0} and integration by parts,
\be \label{pf_reg_3}
   \|\tau\|_{H^{-1/2}(\G)}^2 \simeq \<\tau,\CV\tau\>_\G
   = -\<v',\CV\tau\>_\CT = \jump{v}\cdot (\CV\tau(x_j))_{j=1}^N
   \lesssim |\jump{v}| |\one| \|\CV\tau\|_{C^0(\G)}.
\ee
By the Sobolev embedding theorem and the continuity of $\CV$
(see \eqref{V_cont} with $s=1/2$) there holds
\be \label{pf_reg_4}
   \|\CV\tau\|_{C^0(\G)} \lesssim \|\CV\tau\|_{H^1(\G)}
   \lesssim \|\tau\|_{L^2(\G)}.
\ee
Combination of \eqref{pf_reg_1}--\eqref{pf_reg_4} shows that
\[
   \|\tau\|_{L^2(\G)}^2 \lesssim \|\tau\|_{H^{-1/2}(\G)}^2
   \lesssim |\jump{v}| |\one| \|\tau\|_{L^2(\G)}
   = \sqrt{N} |\jump{v}| \|\tau\|_{L^2(\G)}.
\]
This estimate yields also an estimate for $v$:
\[
   \|v'\|_{L^2(\CT)} = \|\tau\|_{L^2(\G)} \lesssim \sqrt{N} |\jump{v}|.
\]
This proves the lemma.
\end{proof}

\begin{lemma} \label{la_PF}
There holds
\[
   \|v\|_{L^2(\G)}
   \lesssim
   \|v'\|_{L^2(\CT)} + \sqrt{N} |\jump{v}| + |\<v,1\>_\G|
   \qquad\forall v\in H^1(\CT).
\]
\end{lemma}

\begin{proof}
For given $v\in H^1(\CT)$ define $v_0:=v-c_v$ with $c_v:=|\G|^{-1}\<v,1\>_\G$
so that $\<v_0,1\>_\G=0$. Then let $\psi\in H^{1/2}(\G)/\R$ be the solution
of $\CW\psi=v_0$. In particular, $v_0\in L^2(\G)$ so that $\psi\in H^1(\G)$
with bound
\be \label{reg_psi}
   |\psi|_{H^{1/2}(\G)} + |\psi|_{H^1(\G)} \lesssim \|v_0\|_{L^2(\G)},
\ee
see \cite[Theorem 3]{Costabel_88_BIO}. Now, using the definition of $\psi$,
relation \eqref{WV}, integrating by parts, using a duality estimate,
the continuities of $\CV$ \eqref{V_cont} and
$(\cdot)':\;H^{1/2}(\G)\to H^{-1/2}(\G)$, a quotient space argument,
and the Sobolev embedding theorem, we obtain
\begin{align*}
   \|v_0\|_{L^2(\G)}^2
   &=
   \<\CW\psi,v_0\>_\G
   =
   -\<(\CV\psi')',v_0\>_\G
   =
   \<\CV\psi',v_0'\>_\CT + \jump{v_0}\cdot\bigl(\CV\psi'(x_j)\bigr)_{j=1}^N
   \\
   &\lesssim
   \|\CV\psi'\|_{H^{1/2}(\G)} \|v_0'\|_{H^{-1/2}(\CT)}
   + |\jump{v_0}|\; \Bigl|\bigl(\CV\psi'(x_j)\bigr)_{j=1}^N\Bigr|
   \\
   &\lesssim
   \|\psi'\|_{H^{-1/2}(\G)} \|v_0'\|_{L^2(\CT)}
   + |\jump{v_0}| \sqrt{N}\|\CV\psi'\|_{C^0(\G)}
   \\
   &\lesssim
   |\psi|_{H^{1/2}(\G)} \|v_0'\|_{L^2(\CT)}
   + \sqrt{N} |\jump{v_0}| \|\CV\psi'\|_{H^1(\G)}
   \\
   &\lesssim
   |\psi|_{H^{1/2}(\G)} \|v_0'\|_{L^2(\CT)}
   + \sqrt{N} |\jump{v_0}| \|\psi'\|_{L^2(\G)}.
\end{align*}
By \eqref{reg_psi} this yields
\[
   \|v_0\|_{L^2(\G)}
   \lesssim
   \|v_0'\|_{L^2(\CT)} + \sqrt{N} |\jump{v_0}|.
\]
Note that $\jump{v_0}=\jump{v}$ and $v'=v_0'$ in $L^2(\CT)$.
Therefore, this bound together with
\[
   \|v\|_{L^2(\G)} \le \|v_0\|_{L^2(\G)} + \|c_v\|_{L^2(\G)}
                   \lesssim \|v_0\|_{L^2(\G)} + |\<v,1\>_\G|
\]
proves the statement.
\end{proof}
\subsection{Norm equivalences} \label{sec_equiv}

We now prove norm equivalences in our test space $V$, comparing
the optimal test norm $\|\cdot\|_{V,{\rm opt},\alpha}$ \eqref{Vopt} with the more
practical norms $\|\cdot\|_{V,i}$ ($i=1,2$), \eqref{V1}, \eqref{V2}.
As a consequence (Corollary~\ref{cor_equiv})
we obtain lower and upper bounds for the corresponding energy norms
in the trial space $U$ which will be needed
to analyze the DPG method with optimal test functions based on the inner products
defining the norms $\|\cdot\|_{V,i}$ ($i=1,2$).

\begin{theorem} \label{thm_equiv1}
There hold the estimates
\be \label{equiv1_1}
   \|(\tau,v)\|_{V,{\rm opt},\alpha} \lesssim \|(\tau,v)\|_{V,1}
   \qquad\forall (\tau, v)\in V
   \qquad (\alpha = h_{\min}^{-1/2}),
\ee
\be \label{equiv1_2}
   \|(\tau,v)\|_{V,1} \lesssim \|(\tau,v)\|_{V,{\rm opt},\alpha}
   \qquad\forall (\tau, v)\in V
   \qquad (\alpha = N^{-1/2}).
\ee
Here, $h_{\min}$ is the length of the shortest element of the mesh $\CT$.
\end{theorem}

\begin{proof}
To prove \eqref{equiv1_1} we have to show that
\be \label{pf_equiv1_1b}
   \|(\CV\tau)'\|_{L^2(\G)} + \|\tau+v'\|_{L^2(\CT)}
   + \alpha^{-1} |\jump{v}| + |\<v,1\>_\G|
   \lesssim
   \|\tau\|_{L^2(\G)} + \|v\|_{H^1(\CT)}
\ee
holds for any $\tau\in L^2(\G)$ and $v\in H^1(\CT)$.
The first two terms are estimated by using the triangle inequality
and the continuity of $\CV$ (see \eqref{V_cont} with $s=1/2$),
\[
   \|(\CV\tau)'\|_{L^2(\G)} \le \|\CV\tau\|_{H^1(\G)} \lesssim \|\tau\|_{L^2(\G)}
   \qquad\forall\tau\in L^2(\G),
\]
and by the Cauchy-Schwarz inequality we have
\[
   |\<v,1\>_\G| \le \|1\|_{L^2(\G)} \|v\|_{L^2(\G)}
   \qquad\forall v\in L^2(\G).
\]
To show \eqref{pf_equiv1_1b} it therefore remains to estimate $|\jump{v}|$.

By the Sobolev embedding theorem on a reference element $\hat T$
(i.e., the continuous inclusion $H^1(\hat T)\subset C^0(\hat T)$) and
scaling, we obtain for any $T\in\CT$
\be \label{pf_equiv1_1}
   |v(x)|^2 \lesssim |T|^{-1} \|v\|_{L^2(T)}^2 + |T| |v|_{H^1(T)}^2
   \qquad \forall x\in T, v\in H^1(T).
\ee
Therefore, together with the triangle inequality,
\[
   |\jump{v}|^2
   \lesssim
   \sum_{T\in\CT} \Bigl(|T|^{-1} \|v\|_{L^2(T)}^2 + |T| |v|_{H^1(T)}^2\Bigr)
   \lesssim
   h_{\min}^{-1} \|v\|_{H^1(\CT)}^2
   \quad\forall v\in H^1(\CT).
\]
Therefore, this term can be bounded as required for
\eqref{pf_equiv1_1b} if $\alpha\gtrsim h_{\min}^{-1/2}$.

To prove \eqref{equiv1_2} we have to show that
\be \label{pf_equiv1_2b}
   \|\tau\|_{L^2(\G)} + \|v\|_{H^1(\CT)}
   \lesssim
   \|(\CV\tau)'\|_{L^2(\G)} + \|\tau+v'\|_{L^2(\CT)}
   + \alpha^{-1} |\jump{v}| + |\<v,1\>_\G|
\ee
holds for any $\tau\in L^2(\G)$ and $v\in H^1(\CT)$.
To this end let $\tau\in L^2(\G)$ and $v\in H^1(\CT)$ be given.
We select $g_1:=\tau+v'\in L^2(\CT)$ and $g_2:=(\CV\tau)'\in L^2(\G)$
and denote by $(\tau_c,v_c)\in L^2(\G)\times H^1(\G)$ a solution of
\eqref{vtau_c}. Then $(\tau_0,v_0):=(\tau,v)-(\tau_c,v_c)$
solves \eqref{vtau_0} and there holds $\jump{v_0}=\jump{v}$.

By Lemmas~\ref{la_c} and \ref{la_0} we find that there holds
\[
   \|\tau_c\|_{L^2(\G)} + |v_c|_{H^1(\G)}
   \lesssim
   \|g_1\|_{L^2(\G)} + \|g_2\|_{L^2(\G)}
   =
   \|\tau+v'\|_{L^2(\CT)} + \|(\CV\tau)'\|_{L^2(\G)}
\]
and
\[
   \|\tau_0\|_{L^2(\G)} + \|v_0'\|_{L^2(\CT)} \lesssim \sqrt{N} |\jump{v}|.
\]
These two bounds together with the triangle inequality for
$(\tau,v)=(\tau_c,v_c)+(\tau_0,v_0)$ and the $L^2(\G)$-bound
\[
   \|v\|_{L^2(\G)}
   \lesssim
   \|v'\|_{L^2(\CT)} + \sqrt{N} |\jump{v}| + |\<v,1\>_\G|
\]
by Lemma~\ref{la_PF} yield \eqref{pf_equiv1_2b} if
$\sqrt{N}\lesssim\alpha^{-1}$. This finishes the proof of the theorem.
\end{proof}

\begin{theorem} \label{thm_equiv2}
There hold the estimates
\be \label{equiv2_1}
   \|(\tau,v)\|_{V,{\rm opt},\alpha} \lesssim \|(\tau,v)\|_{V,2}
   \qquad\forall (\tau, v)\in V
   \qquad (\alpha = h_{\min}^{-1/2}),
\ee
\be \label{equiv2_2}
   \|(\tau,v)\|_{V,2} \lesssim \|(\tau,v)\|_{V,{\rm opt},\alpha}
   \qquad\forall (\tau, v)\in V
   \qquad (\alpha = N^{-1/2}).
\ee
Here, $h_{\min}$ is the length of the shortest element of the mesh $\CT$.
\end{theorem}

\begin{proof}
To prove \eqref{equiv2_1} we have to show that
\[
   \|(\CV\tau)'\|_{L^2(\G)} + \|\tau+v'\|_{L^2(\CT)}
   + \alpha^{-1} |\jump{v}| + |\<v,1\>_\G|
   \lesssim
   \|\tau\|_{L^2(\G)} + \|v'\|_{L^2(\CT)} + |\hat v|_h
\]
holds for any $\tau\in L^2(\G)$ and $v\in H^1(\CT)$.
By \eqref{pf_equiv1_1b} it is enough to show that
\be \label{pf_equiv2_1b}
   \|v\|_{L^2(\G)}
   \lesssim
   \|v'\|_{L^2(\CT)} + |\hat v|_h
   \qquad\forall v\in H^1(\CT).
\ee
By standard arguments there holds on a reference element $\hat T$
with an endpoint $x_{\hat T}$
\[
   \|w\|_{L^2(\hat T)}^2
   \lesssim
   |w|_{H^1(\hat T)}^2 + |w(x_{\hat T})|^2
   \qquad\forall w\in H^1(\hat T).
\]
Transformation proves that
\[
   |T|^{-1} \|v\|_{L^2(T)}^2
   \lesssim
   |T| |v|_{H^1(T)}^2 + |v(x_T)|^2
   \qquad\forall v\in H^1(T), T\in\CT,
\]
so that by summation,
\be \label{PF2}
   \|v\|_{L^2(\G)}
   \lesssim
   h_{\max} \|v'\|_{L^2(\CT)} + |\hat v|_h
   \qquad\forall v\in H^1(\CT).
\ee
This yields \eqref{pf_equiv2_1b}.

To prove \eqref{equiv2_2} we have to show that
\be \label{pf_equiv2_2b}
   \|\tau\|_{L^2(\G)} + \|v'\|_{L^2(\CT)} + |\hat v|_h
   \lesssim
   \|(\CV\tau)'\|_{L^2(\G)} + \|\tau+v'\|_{L^2(\CT)}
   + \alpha^{-1} |\jump{v}| + |\<v,1\>_\G|
\ee
holds for any $\tau\in L^2(\G)$ and $v\in H^1(\CT)$.
By \eqref{pf_equiv1_2b} it is enough to show that
\be \label{pf_equiv2_2c}
   |\hat v|_h
   \lesssim
   \|v\|_{H^1(\CT)}
   \qquad\forall v\in H^1(\CT).
\ee
By \eqref{pf_equiv1_1} there holds
\[
   |v(x_T)|^2 \lesssim
   |T|^{-1} \|v\|_{L^2(T)}^2 + |T| |v|_{H^1(T)}^2
   \qquad\forall v\in H^1(T), T\in\CT.
\]
Multiplication by $|T|$ and summation over $T\in\CT$ proves that
\[
   |\hat v|_h
   \lesssim
   \|v\|_{L^2(\G)} + h_{\max} \|v'\|_{L^2(\CT)}
   \qquad\forall v\in H^1(\CT),
\]
and this implies \eqref{pf_equiv2_2c}.
The proof of the theorem is finished.
\end{proof}

By duality, norm equivalences in $V$ imply equivalences between corresponding
norms in $U$. This fact is essential to prove error estimates for the DPG method
when using different norms in the test space. The equivalences in $U$ are as
follows.

\begin{corollary} \label{cor_equiv}
Let $\|\cdot\|_{E,i}$ ($i\in\{1,2\}$) be the energy norm defined by duality
\eqref{duality} with test norm
\(
  \|(\tau,v)\|_{V,i}, 
\)
cf.~\eqref{V1}, \eqref{V2}. Then there hold the estimates
\[
   \|(\phi,\sigma,\hat\sigma)\|_{E,i}
   \lesssim
   \|\phi\|_{L^2(\G)} + \|\sigma\|_{L^2(\G)} + h_{\min}^{-1/2} |\hat\sigma|
   \qquad\forall (\phi,\sigma,\hat\sigma)\in U
\]
and
\[
   \|\phi\|_{L^2(\G)} + \|\sigma\|_{L^2(\G)} + N^{-1/2} |\hat\sigma|
   \lesssim
   \|(\phi,\sigma,\hat\sigma)\|_{E,i}
   \qquad\forall (\phi,\sigma,\hat\sigma)\in U.
\]
Here, $h_{\min}$ is the length of the shortest element of the mesh $\CT$.
\end{corollary}

\begin{proof}
Recall the relation \eqref{EU} that identifies the $\|\cdot\|_{U,\alpha}$-norm
as being dual to the $\|\cdot\|_{V,\mathrm{opt},\alpha}$-norm with respect to
$b(\cdot;\cdot)$.
Then the statement of the corollary is equivalent to the statements of
Theorems~\ref{thm_equiv1} and \ref{thm_equiv2} respectively for $i=1$ and $i=2$,
cf. the definition \eqref{U} of the $\|\cdot\|_{U,\alpha}$-norm.
\end{proof}

\section{DPG method with optimal test functions} \label{sec_DPG}
\setcounter{equation}{0} \setcounter{figure}{0} \setcounter{table}{0}

In this section we briefly introduce the DPG method with optimal
test functions for our model problem.

We consider a discrete subspace $U_{hp}\subset U$,
\[
   U_{hp} := U_{hp}^0\times U_{hp}^0\times \R^N,
\]
where $U_{hp}^0\subset L^2(\G)$ is the piecewise polynomial space
\[
   U_{hp}^0 =
   \{\varphi\in L^2(\G);\;
     \varphi|_T \text{ is a polynomial of degree } {p_T}\ \forall T\in\CT\}.
\]
Here, we have simply used the same polynomial degrees $p_T$ for the approximation of
$\phi$ and $\sigma$, and $hp$ refers to the fact that the mesh $\CT$ and polynomial
degrees may vary.

Denoting a basis of $U_{hp}$ by $\{\Bu_i;\; i=1,\ldots,\mathrm{dim}(U_{hp})\}$,
the corresponding space of optimal test functions $V_{hp}\subset V$ is
$\mathrm{span}\{\T(\Bu_i);\; i=1,\ldots,\mathrm{dim}(U_{hp})\}$ with trial-to-test
operator $\T$ defined by
\be \label{Theta}
   \T(\Bu_i)\in V:\quad
   \<\T(\Bu_i), \Bv\>_V = b(\Bu_i, \Bv)\quad\forall \Bv\in V.
\ee
The inner product $\<\cdot,\cdot\>_V$ has to be chosen accordingly to the choice
of test norm, \eqref{V1} or \eqref{V2}.

The DPG method with optimal test functions for the model problem \eqref{model} then is:
{\em Find $(\phi_{hp}, \sigma_{hp}, \hat\sigma_{hp})\in U_{hp}$ such that}
\be \label{DPG}
   b(\phi_{hp},\sigma_{hp},\hat\sigma_{hp};\tau,v) = \<f,v\>_\G
   \qquad\forall (\tau,v)\in V_{hp}.
\ee
By design of the method one obtains optimal convergence \cite{DemkowiczG_11_CDP}:
\be \label{opt}
   \|(\phi-\phi_{hp},\sigma-\sigma_{hp},\hat\sigma-\hat\sigma_{hp})\|_E
   =
   \inf_{(\varphi,\rho,\hat\rho)\in U_{hp}}
   \|(\phi-\varphi,\sigma-\rho,\hat\sigma-\hat\rho)\|_E.
\ee
Here, $\|\cdot\|_E$ is the energy-norm that corresponds to the chosen norm in $V$,
cf.~\eqref{duality}. Now, using the norm equivalences from Corollary~\ref{cor_equiv},
the best approximation property \eqref{opt} immediately implies the following error estimate
which is the second main result of this paper.

\begin{theorem} \label{thm_DPG}
Let $(\phi,\sigma,\hat\sigma)\in U$ and $(\phi_{hp},\sigma_{hp},\hat\sigma_{hp})\in U_{hp}$
be the solutions of the ultra-weak formulation \eqref{ultra} and the DPG-scheme \eqref{DPG},
respectively. Then, if the optimal test functions are calculated by using an inner product
in $V$ corresponding to $\|\cdot\|_{V,1}$ or $\|\cdot\|_{V,2}$, see \eqref{V1}, \eqref{V2},
there holds
\[
   \|\phi-\phi_{hp}\|_{L^2(\G)} + \|\sigma-\sigma_{hp}\|_{L^2(\G)}
                                + N^{-1/2} |\hat\sigma-\hat\sigma_{hp}|
   \lesssim
   \inf_{\varphi\in U_{hp}^0,\rho\in U_{hp}^0}\Bigl(
      \|\phi-\varphi\|_{L^2(\G)} + \|\sigma-\rho\|_{L^2(\G)}
   \Bigr).
\]
\end{theorem}

\subsection{Optimal test functions} \label{sec_test}

In this section we give some details on the calculation of optimal test functions.

As previously mentioned, optimal test functions can be partially calculated
analytically when selecting the inner product $\<\cdot,\cdot\>_V$ in $V$
that induces the norm $\|\cdot\|_{V,2}$, cf.~\eqref{V2}.
Recalling \eqref{Euclid} we define
\[
   \<v, \delta_v\>_h := \sum_{T\in\CT} |T|\; v|_T(x_T) \delta_v|_T(x_T),
   \quad v, \delta_v\in H^1(\CT).
\]
Then the corresponding inner product in $V$ is
\be \label{ip}
  \<(\tau,v), (\delta_\tau, \delta_v)\>_V
  :=
  \<\tau, \delta_\tau\>_\G + \<v', \delta_v'\>_\CT + \<v, \delta_v\>_h
   \quad \tau, \delta_\tau\in L^2(\G),\ v, \delta_v\in H^1(\CT).
\ee
Being $\{\varphi_i;\; i=1,\ldots,\mathrm{dim}(U_{hp}^0)\}$ a basis of $U_{hp}^0$
and denoting by $e_j\in\R^N$ an element of the canonical basis of $\R^N$,
a basis of $U_{hp}$ is given by
\be \label{basis}
   \{(\varphi_i,0,0), (0,\varphi_i,0), (0,0,e_j);\;
     i=1,\ldots,\mathrm{dim}(U_{hp}^0),\; j=1,\ldots,N\}.
\ee
Calculating optimal test functions therefore requires to consider three different
types, as follows.

\paragraph{Calculation of $(\tau,v)=\T(\varphi_i,0,0)$.}
In this case \eqref{Theta} reduces to find $\tau\in L^2(\G)$ and $v\in H^1(\CT)$
such that
\be \label{Theta_phi}
   \<\tau, \delta_\tau\>_\G + \<v', \delta_v'\>_\CT + \<v, \delta_v\>_h
   =
   b(\varphi_i,0,0;\delta_\tau,\delta_v)
   =
   \<\varphi_i, (\CV\delta_\tau)'\>_\G + \<\varphi_i,1\>_\G \<\delta_v,1\>_\G
\ee
for any $\delta_\tau\in L^2(\G)$ and $\delta_v\in H^1(\CT)$.

Selecting $\delta_\tau=0$ and $\delta_v$ as the characteristic function on an
arbitrary element $T\in\CT$ we find that
\[
   v|_T(x_T) = \<\varphi_i,1\>_\G.
\]
Therefore, again selecting $\delta_\tau=0$ and now an arbitrary function $\delta_v$
with support in $\bar T$, we deduce that
\[
   \<v', \delta_v'\>_T
   =
   \<\varphi_i,1\>_\G \bigl(\<\delta_v,1\>_T - |T| \delta_v(x_T)\bigr)
   \quad\forall \delta_v\in H^1(T).
\]
Integration by parts reveals that (with $s$ being the arc length and $s=0$
corresponding to the ``left'' endpoint of $T$) $v$ is the quadratic polynomial
\be \label{v_phi}
   v(s) = \<\varphi_i,1\>_\G \bigl( (|T|-s/2)s + 1 \bigr)
   \quad\text{on}\ T\simeq s\in (0,|T|).
\ee
Selecting $\delta_v=0$ in \eqref{Theta_phi} leads to
\be \label{tau_phi}
   \tau\in L^2(\G):\quad
   \<\tau, \delta_\tau\>_\G = \<\varphi_i, (\CV\delta_\tau)'\>_\G
   \quad\forall \delta_\tau\in L^2(\G).
\ee
This problem can be easily approximated by finite elements. Then the system matrix
is a simple block-diagonal mass matrix. Global properties of the boundary integral
operator of the model problem are inherited via the right-hand side of \eqref{tau_phi}.
In general an optimal test function $\tau$ will be non-zero also outside the
support of a basis function $\varphi_i$. In Section~\ref{sec_int} we give some
more details.

\paragraph{Calculation of $(\tau,v)=\T(0,\varphi_i,0)$.}
Now \eqref{Theta} reduces to find $\tau\in L^2(\G)$ and $v\in H^1(\CT)$
such that
\be \label{Theta_sigma}
   \<\tau, \delta_\tau\>_\G + \<v', \delta_v'\>_\CT + \<v, \delta_v\>_h
   =
   \<\varphi_i, \delta_\tau+\delta_v'\>_\CT
\ee
for any $\delta_\tau\in L^2(\G)$ and $\delta_v\in H^1(\CT)$.
One immediately concludes that
\be \label{tau_sigma}
   \tau=\varphi_i.
\ee
To determine $v$ we assume that $\varphi_i$ has support in a single element
$\bar T$. Then, selecting $\delta_\tau=0$ and $\delta_v$ with support in $\bar T$
in \eqref{Theta_sigma}, we find that
\[
   \<v', \delta_v'\>_T + |T|\; v|_T(x_T) \delta_v(x_T)
   = \<\varphi_i, \delta_v'\>_T \quad\forall \delta_v\in H^1(T).
\]
Integrating by parts and identifying the strong form of a second-order equation
for $v$ leads to the solution
\be \label{v_sigma}
   \mathrm{supp}(\varphi_i)\subset \bar T\Rightarrow\qquad
   v(s) = \int_0^s \varphi_i(x(t))\,dt
   \quad\text{on}\quad T\simeq s\in (0,|T|)
\ee
and $v=0$ elsewhere.
As before, $s$ is the arc-length and for simplicity $s=0$ corresponds to the
``left'' endpoint of $T$.

\paragraph{Calculation of $(\tau,v)=\T(0,0,e_j)$.}
We have to determine $\tau\in L^2(\G)$ and $v\in H^1(\CT)$
such that
\be \label{Theta_hatsigma}
   \<\tau, \delta_\tau\>_\G + \<v', \delta_v'\>_\CT + \<v, \delta_v\>_h
   =
   e_j\cdot\jump{\delta_v}
\ee
for any $\delta_\tau\in L^2(\G)$ and $\delta_v\in H^1(\CT)$.
We find that $\tau=0$.

Let us assume that $T_{j-1}$ and $T_j$ are the elements that have,
in mathematically positive orientation, the nodes $x_{j-1}$, $x_j$ and
$x_j$, $x_{j+1}$, respectively.
Selecting $\delta_v$ in \eqref{Theta_hatsigma} with support in $\bar T_{j-1}$
respectively $\bar T_j$ we obtain the two relations
\[
   - \<v'', \delta_v\>_{T_{j-1}}
   + v'|_{T_{j-1}}(x_j)    \delta_v(x_j)
   - v'|_{T_{j-1}}(x_{j-1})\delta_v(x_{j-1})
   + |T_{j-1}|\; v|_{T_{j-1}}(x_{j-1}) \delta_v(x_{j-1})
   = -\delta_v(x_j)
\]
for any $\delta_v\in H^1(T_{j-1})$ and
\[
   - \<v'', \delta_v\>_{T_{j}}
   + v'|_{T_{j}}(x_{j+1})    \delta_v(x_{j+1})
   - v'|_{T_{j}}(x_j)\delta_v(x_j)
   + |T_{j}|\; v|_{T_{j}}(x_j) \delta_v(x_j)
   = \delta_v(x_j)
\]
for any $\delta_v\in H^1(T_{j})$. The solution of the latter relation
is $v=|T_{j}|^{-1}$. We rewrite the first relation in strong form
\[
   v''=0\quad \text{on}\quad T_{j-1},\quad
   -v'(x_{j-1}) + |T_{j-1}|\; v(x_{j-1}) = 0, \quad
   v'(x_j)=-1,
\]
to find a linear function on $T_{j-1}$ as solution. Together, identifying
$x_{j-1}$ with arc length $s=0$,
\be \label{v_hatsigma}
   v(s) = \left\{
   \begin{array}{cl}
      -s-|T_{j-1}|^{-1} & \text{on}\ T_{j-1} \simeq s\in (0,|T_{j-1}|),\\
      |T_{j}|^{-1}      & \text{on}\ T_{j},\\
      0                 & \text{otherwise}.
   \end{array}\right.
\ee
In Table~\ref{tab_test} we give an overview of the findings in the
form
\[
   \T:\; U_{hp}\subset U=L^2(\G)\times L^2(\G)\times \R^N\to
        V=L^2(\G)\times H^1(\CT);\quad
   (\tau,v)=\T(\phi,\sigma,\hat\sigma).
\]

\begin{table}
\begin{center}
\begin{tabular}{c|l}
trial basis function & optimal test function\\
$(\phi,\sigma,\hat\sigma)$ &  $(\tau,v)=\T(\phi,\sigma,\hat\sigma)$\\
\hline
$(\varphi_i, 0, 0)$ &
$\rule{0em}{1.7em}\tau\in L^2(\G):\;
   \left\{\begin{array}{l}
      \<\tau, \delta_\tau\>_\G = \<\varphi_i, (\CV\delta_\tau)'\>_\G\
      \forall\delta_\tau\in L^2(\G)\ \eqref{tau_phi}\\
      \text{must be approximated in practice}
   \end{array}\right.$
\\
& $v$ piecewise quadratic polynomial \eqref{v_phi}
\\ \hline
$\rule{0em}{1.2em}(0, \varphi_i, 0)$ &
$\tau=\varphi_i$ \eqref{tau_sigma}
\\
& $v = \left\{\begin{array}{l}
       \text{antiderivative of $\varphi_i$ on $\mathrm{supp}(\varphi_i)\subset \bar T$},\\
       0\ \text{elsewhere}
       \end{array}\right.$
       \eqref{v_sigma}
\\ \hline
$\rule{0em}{1.2em}(0, 0, e_j)$ &
$\tau=0$
\\
& $v = \left\{\begin{array}{l}
       \text{piecewise linear on patch containing $x_j$}, \\
       0\ \text{elsewhere}
       \end{array}\right.$
       \eqref{v_hatsigma}
\end{tabular}
\end{center}
\caption{Summary optimal test functions for closed curve}
\label{tab_test}
\end{table}

\section{DPG setting on an open curve} \label{sec_int}
\setcounter{equation}{0} \setcounter{figure}{0} \setcounter{table}{0}

So far we have considered the model problem \eqref{model} on a closed curve.
For simplicity we report on numerical results on an open curve.
Here the setting is slightly different since then the energy space
is such that the solution is unique without considering a quotient space.
In the following we describe the setting and calculation of optimal test functions
in this case.

For an open polygon the model problem is: {\em For given
$f\in L^2(\G)$ find $\phi\in \tilde H^{1/2}(\G)$ such that}
\[
   \<\CW\phi,\psi\>_\G = \<f, \psi\>_\G\quad\forall \psi\in \tilde H^{1/2}(\G).
\]
Here, $\<\cdot,\cdot\>_\G$ denotes again a duality, this time between
$H^{-1/2}(\G)$ and $\tilde H^{1/2}(\G)$. The latter space consists of the jumps
across $\G$ of $H^1$ functions in the exterior of $\G$ and $H^{-1/2}(\G)$ is its
dual space. As before we rewrite this problem as the first order system
\eqref{system} with additional unknown $\sigma=\CV\phi'$.

Let $\CT$ be a mesh of elements $T$ on $\G$ that is compatible with
the corners, and let us denote the nodes by $x_j$, $j=1,\ldots,N$,
$x_1$ and $x_N$ being the endpoints.
Proceeding as before in the case of a closed polygon
we obtain the ultra-weak formulation:
{\em find $\sigma\in L^2(\G)$, $\phi\in L^2(\G)$ and $\hat\sigma\in\R^N$
such that}
\begin{alignat}{5}
   \<\sigma, \tau\>_\G &+ \<\phi, (\CV\tau)'\>_\G &&= 0
      &\qquad&\forall\tau\in L^2(\G),
   \label{ultra1_int}\\
   \<\sigma, v'\>_\CT &+ \quad\hat\sigma\cdot\jump{v}
      &&= \<f, v\>_\G
      &\qquad&\forall v\in H^1(\CT).
   \label{ultra2_int}
\end{alignat}
Here, $\hat\sigma\in\R^N$ with $\hat\sigma_j=\sigma(x_j)$ ($j=1,\ldots, N$)
replaces the nodal values of $\sigma$ and we have used that
$\phi$ vanishes at the endpoints of $\G$.
Furthermore, at interior nodes $x_j$ ($j=2,\ldots,N-1$) the jump $\jump{v}_j$
is defined as before, at the endpoints we re-define
$\jump{v}_1:=v(x_1)$, $\jump{v}_N:=-v(x_N)$, and, as before,
\[
   \jump{v} := (\jump{v}_j)_{j=1}^N\in\R^N \qquad\text{for}\quad v\in H^1(\CT).
\]
In short form we have arrived at the ultra-weak formulation \eqref{ultra} with
\[
   b(\phi,\sigma,\hat\sigma;\tau,v)
   := \<\phi, (\CV\tau)'\>_\G + \<\sigma, \tau+v'\>_\CT + \hat\sigma\cdot\jump{v}.
\]
Note that the rank-one terms $\<\phi,1\>_\G$ and $\<v,1\>_\G$ have disappeared.
Indeed, system \eqref{ultra1_int}, \eqref{ultra2_int} does not have a kernel
with respect to $\phi$ or $v$.

\begin{remark}
An open curve is the extreme case where the solution $\phi$ is usually not
in $H^1(\G)$ even for smooth right-hand side functions. Therefore, our technique
of proving existence and uniqueness of a solution to the ultra-weak formulation and
equivalence with the model problem (Theorem~\ref{thm_ultra}) does not apply
without major changes.
For the same reason, the theory of Section~\ref{sec_norms} on the equivalence
of norms is just outside the range of open curves. Our numerical experiments will
show that the DPG method with optimal test functions delivers satisfactory
results even in this limit case.
\end{remark}

\subsection{Optimal test functions} \label{sec_int_test}

The calculation of optimal test functions on open curves is almost identical
to the procedure presented for closed curves in Section~\ref{sec_test}.
To be able to calculate most test functions analytically,
we again select the inner product \eqref{ip}
that defines the norm $\|\cdot\|_{V,2}$ in $V$ \eqref{V2}.
The ansatz space is defined as before, and denoted by $U_{hp}$ with basis
as in \eqref{basis}.

Let us recall the calculation of test functions and indicate necessary changes.

\paragraph{Calculation of $(\tau,v)=\T(\varphi_i,0,0)$.}
In this case the new bilinear form $b(\cdot,\cdot)$ with missing rank-one terms
induces some changes. Instead of \eqref{Theta_phi} we now have
\[
   \<\tau, \delta_\tau\>_\G + \<v', \delta_v'\>_\CT + \<v, \delta_v\>_h
   =
   \<\varphi_i, (\CV\delta_\tau)'\>_\G.
\]
Selecting $\delta_\tau=0$ and $\delta_v$ arbitrary reveals that $v=0$.
On the other, as before $\tau$ is defined by
\be \label{tau_phi_int}
   \tau\in L^2(\G):\quad
   \<\tau, \delta_\tau\>_\G = \<\varphi_i, (\CV\delta_\tau)'\>_\G
   \quad\forall \delta_\tau\in L^2(\G).
\ee
For our numerical experiments, we approximate $\tau$ by the finite element
solution $\tilde\tau\in V_{h\tilde p}^0$ of
\be \label{tau_phi_fem}
   \tilde\tau\in V_{h\tilde p}^0:\quad
   \<\tilde\tau, \delta_\tau\>_\G = \<\varphi_i, (\CV\delta_\tau)'\>_\G
   \quad\forall \delta_\tau\in V_{h\tilde p}^0.
\ee
Here, the discrete space
\[
   V_{h\tilde p}^0 :=
   \{\varphi\in L^2(\G);\;
     \varphi|_T \text{ is a polynomial of degree } {\tilde p_T}\ \forall T\in\CT\}.
\]
is based on the same mesh $\CT$ but with increased polynomial
degrees $\tilde p_T>p_T$.
The right-hand side of \eqref{tau_phi_fem} can be implemented by integrating by
parts and using standard procedures for stiffness matrices involving the operator $\CV$,
see, e.g., \cite{ErvinHS_93_hpB}.

\paragraph{Calculation of $(\tau,v)=\T(0,\varphi_i,0)$.}
In this case the problem \eqref{Theta_sigma} does not change so that
optimal test functions depending on $\sigma=\varphi_i$ do not change.
We have $\tau=\varphi_i$ and
$v(s) = \int_0^s \varphi_i(x(t))\,dt$ on $T\simeq s\in (0,|T|)$
and $v=0$ elsewhere, if $\mathrm{supp}(\varphi_i)\subset \bar T$.

\paragraph{Calculation of $(\tau,v)=\T(0,0,e_j)$.}
In this case the only change is due to the jump $\jump{\delta_v}$ reducing
to plus or minus the trace of $\delta_v$ at the endpoints of $\G$.
More precisely, $\tau=0$ as determined previously and
\begin{align}
   \label{v_hatsigma_int1}
   v(s) &= \left\{
   \begin{array}{cl}
      |T_{j}|^{-1}      & \text{on}\ T_{j}\\
      0                 & \text{otherwise}
   \end{array}\right. &&(j=1),
   \\
   \label{v_hatsigma_intj}
   v(s) &= \left\{
   \begin{array}{cl}
      -s-|T_{j-1}|^{-1} & \text{on}\ T_{j-1} \simeq s\in (0,|T_{j-1}|)\\
      |T_{j}|^{-1}      & \text{on}\ T_{j}\\
      0                 & \text{otherwise}
   \end{array}\right. &&(j=2,\ldots,N-1),
   \\
   \label{v_hatsigma_intN}
   v(s) &= \left\{
   \begin{array}{cl}
      -s-|T_{j-1}|^{-1} & \text{on}\ T_{j-1} \simeq s\in (0,|T_{j-1}|)\\
      0                 & \text{otherwise}
   \end{array}\right. &&(j=N).
\end{align}

A summary of optimal test functions, corresponding to Table~\ref{tab_test},
is given in Table~\ref{tab_test_int}.

\begin{table}
\begin{center}
\begin{tabular}{c|l}
trial basis function & optimal test function\\
$(\phi,\sigma,\hat\sigma)$ &  $(\tau,v)=\T(\phi,\sigma,\hat\sigma)$
\hspace{10em} \\
\hline
$(\varphi_i, 0, 0)$ &
   $\rule{0em}{1.7em}\tau\in L^2(\G):\;
   \left\{\begin{array}{ll}
          \<\tau, \delta_\tau\>_\G = \<\varphi_i, (\CV\delta_\tau)'\>_\G,\
          \forall\delta_\tau\in L^2(\G) &\eqref{tau_phi_int}\\
          \text{approximated by}\ \tau_{h\tilde p} &\eqref{tau_phi_fem}
   \end{array}\right.$
\\
& $v=0$
\\ \hline
$\rule{0em}{1.2em}(0, \varphi_i, 0)$ &
$\tau=\varphi_i$ \eqref{tau_sigma}
\\
& $v = \left\{\begin{array}{l}
       \text{antiderivative of $\varphi_i$ on $\mathrm{supp}(\varphi_i)\subset \bar T$},\\
       0\ \text{elsewhere}
       \end{array}\right.$
       \eqref{v_sigma}
\\ \hline
$\rule{0em}{1.2em}(0, 0, e_j)$ &
$\tau=0$
\\
& $v = \left\{\begin{array}{l}
       \text{piecewise linear on patch containing $x_j$}, \\
       0\ \text{elsewhere}
       \end{array}\right.$
       \eqref{v_hatsigma_int1}--\eqref{v_hatsigma_intN}
\end{tabular}
\end{center}
\caption{Summary optimal test functions for open curve}
\label{tab_test_int}
\end{table}

\subsection{Numerical results} \label{sec_num}

We consider the model problem \eqref{model} on the open curve $\G=(-1,1)\times\{0\}$
with $f=1/2$. In this case the exact solution is $\phi(x_1,x_2)=\sqrt{1-x_1^2}$.
It has the well-known square root singularities at the endpoints of $\G$.
In particular, $\phi\in H^s(\G)$ for any $s<1$. Therefore, for piecewise polynomial
approximations of $\phi$ and $\sigma$ on quasi-uniform meshes one expects an
$L^2$-error of the order $h\sim N^{-1}$
($h$ being the mesh size and $N$ the dimension of the discrete space).
On the other hand, for optimally refined meshes one expects an order
$N^{-(p+1)}$ when piecewise polynomials of degree $p$ are used.
Precisely these orders are observed for uniform and adaptively refined meshes
in Figures~\ref{fig_uniform} and \ref{fig_adaptive}, respectively, for degrees
$p=0,1,2$. We thus confirm the error estimate by Theorem~\ref{thm_DPG} (proved
for closed curves). For each case the $L^2$ and energy errors are plotted on a
double-logarithmic scale. Both curves (for $L^2$ and energy error) are roughly
parallel. This confirms our theoretical result of equivalence of both norms
(proved by Corollary~\ref{cor_equiv} for closed curves).

Let us make a few comments on the implementation.
The test functions are as described in Section~\ref{sec_int_test}.
Problem \eqref{tau_phi_fem} is solved by using
the corresponding mesh of the ansatz space and by increasing polynomial degrees
only by one. Here, the weakly-singular operator is implemented by analytical
inner integration and outer Gauss quadrature, cf.~\cite{ErvinHS_93_hpB} for details.
To calculate the energy errors we approximate the trial-to-test operator
$\Theta$ analogously, but increasing the polynomial degrees by two.
The adaptive refinement procedure uses the natural error indicators which are
the local contributions of the energy error on the individual elements. We halve
elements of the largest indicators and which, as a whole, are associated to 50\%\
of the total energy error.

\begin{figure}[htb]
\centering
\includegraphics[width=0.7\textwidth]{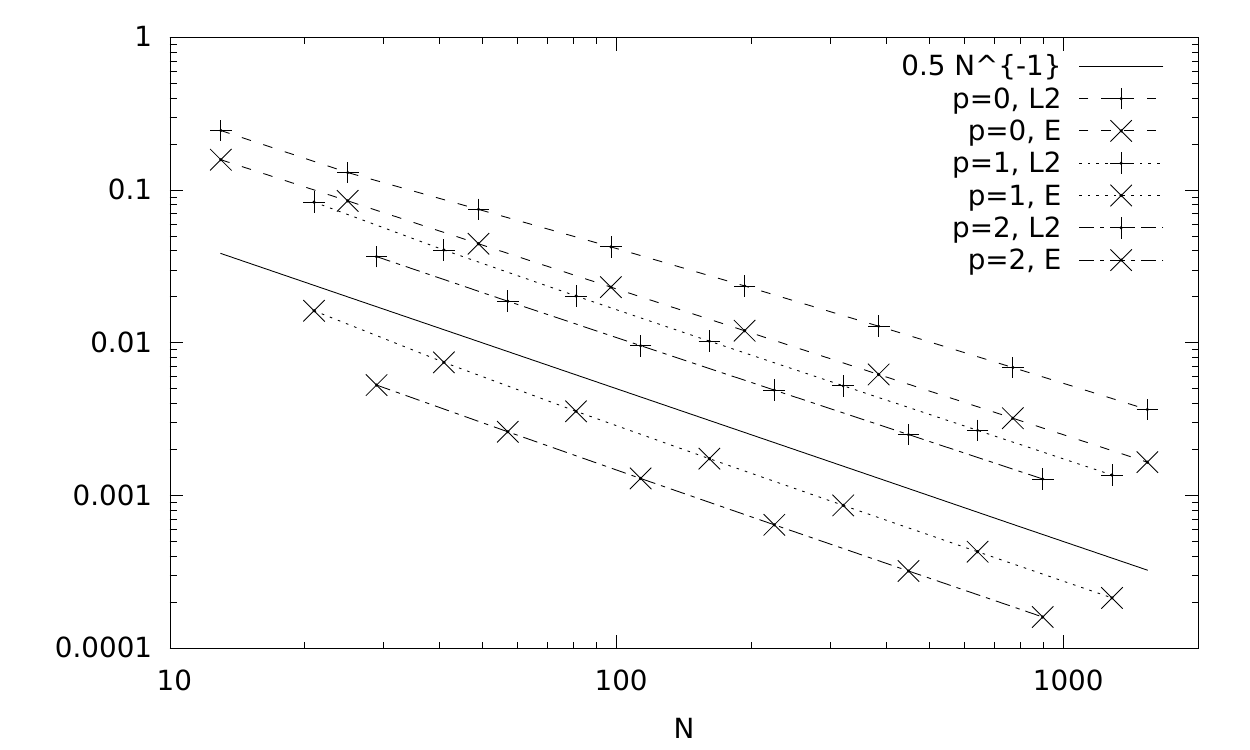} 
\caption{Uniform refinement with $p=0,1,2$ (errors in $L^2$ ``L2'' and energy norm ``E'').}
\label{fig_uniform}
\end{figure}

\begin{figure}[htb]
\centering
\includegraphics[width=0.7\textwidth]{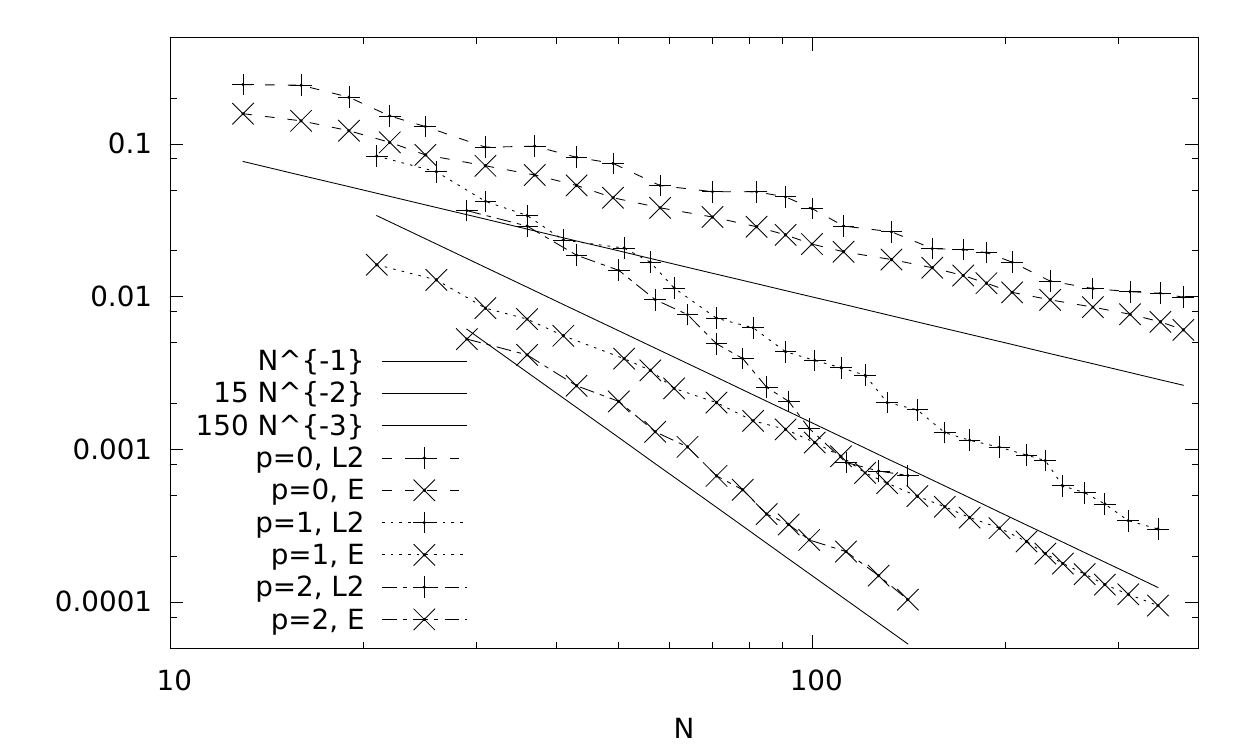} 
\caption{Adaptive refinement with $p=0,1,2$ (errors in $L^2$ ``L2'' and energy norm ``E'').}
\label{fig_adaptive}
\end{figure}

\section{Conclusions}

We have proposed an ultra-weak formulation for hypersingular integral equations
on open and on closed curves. In the case of closed curves, we have proved its
well-posedness and equivalence to the standard variational formulation.
Based on this ultra-weak form we have defined a DPG method with optimal
test functions and, for closed curves, have shown its quasi-optimal convergence
in $L^2$ as well as the energy norm. The proof makes use of the equivalence
of both norms. This equivalence as well as the quasi-optimal convergence
are verified by numerical experiments on an open curve. Error calculation is
inherent to the method by construction. Since appearing test norms are of local
nature, the construction of error indicators and adaptive strategies is
straightforward. This has also been confirmed by numerical experiments.

Overall, the DPG method with optimal test functions, based on an ultra-weak
formulation, has shown to be a fully functional method also for boundary
integral equations (at least for the model problem) with the great advantage
of being posed in local norms, contrary to complicated fractional-order Sobolev
norms in standard Galerkin methods.

\bibliographystyle{siam}
\bibliography{/home/norbert/tex/bib/bib,/home/norbert/tex/bib/heuer,/home/norbert/tex/bib/fem}

\end{document}